\documentclass[11pt]{amsart}
\usepackage[margin=1in]{geometry}
\usepackage[utf8]{inputenc}
\usepackage{amsthm, amssymb}
\usepackage[colorlinks=true, pdfstartview=FitV, linkcolor=blue, citecolor=blue, urlcolor=blue]{hyperref}
\usepackage[colorinlistoftodos]{todonotes}
\usepackage{algorithm}
\usepackage{algpseudocode}
\usepackage{pgf}
\usepackage{tikz}
\usepackage{tikz-cd}
\usetikzlibrary{positioning,shapes,shadows,arrows}
\usepackage{bbm,bm}
\usepackage{enumitem}
\usepackage{mathabx}
\usepackage{comment}
\usepackage{ytableau}
\usepackage{thmtools}
\usepackage{thm-restate}

\newtheorem{prop}{Proposition}[section]
\newtheorem{thm}[prop]{Theorem}

\newtheorem{lemma}[prop]{Lemma}

\newtheorem{conj}[prop]{Conjecture}
\theoremstyle{remark}
\newtheorem{exa}[prop]{Example}
\newtheorem{rem}[prop]{Remark}
\newtheorem{defn}[prop]{Definition}
\newtheorem{prob}[prop]{Problem}

\definecolor{darkblue}{rgb}{0.0,0,0.7} 
\definecolor{darkred}{rgb}{0.7,0,0} 
\definecolor{darkgreen}{rgb}{0, .6, 0} 
\newcommand{\definition}[1]{{\color{darkred}\emph{#1}}} 

\newcommand{\leading}{\mathsf{in}}
\newcommand{\A}{\mathcal{A}}
\newcommand{\B}{\mathcal{B}}
\newcommand{\C}{\mathcal{C}}
\newcommand{\D}{\mathcal{D}}
\newcommand{\OP}{\mathcal{OP}}
\newcommand{\code}{\mathsf{code}}

\title{Constructing a Gr\"{o}bner basis of Griffin's ideal}
\author[T.~Yu]{Tianyi Yu}
\address[T. Yu]{Department of Mathematics, UC San Diego, La Jolla, CA 92093, U.S.A.}
\email{tiy059@ucsd.edu}

\begin{document}

\maketitle

\begin{abstract}
In his Ph.D. thesis, Sean Griffin introduced a family of ideals 
and found monomial bases for their quotient rings.  
These rings simultaneously generalize 
the Delta Conjecture coinvariant rings of Haglund-Rhoades-Shimozono and the cohomology rings of Springer fibers studied by Tanisaki and Garsia-Procesi.
We recursively construct 
a Gr\"{o}bner basis of Griffin's ideals
with respect to the graded reverse lexicographical order.
Consequently, Griffin's monomial basis
is the standard monomial basis.
Coefficients of polynomials in our Gr\"{o}bner basis 
are integers and leading coefficients are one.   
\end{abstract}

\section{Introduction}

For $n \in \mathbb{Z}_{>0}$, a partition $\lambda$, and $s \in \mathbb{Z}_{>0} \sqcup \{\infty\}$ with  $1 \leq |\lambda| \leq n$ and $s \geq \ell(\lambda)$,
Griffin introduced ideals $I_{n, \lambda, s}$ in the polynomial ring $\mathbb{Q}[x_1, \cdots, x_n]$.
Let $R_{n, \lambda, s} := \mathbb{Q}[x_1, \cdots, x_n]/I_{n, \lambda, s}$ be their corresponding quotient ring.
The ideal $I_{n, \lambda,s}$ and the ring
$R_{n,\lambda, s}$
generalize the coinvariant ideal
$I_n := I_{n, (1)^n, \infty}$ 
and the coinvariant ring $R_n := R_{n, (1)^n, \infty}$, 
where $(1)^n$ is
the partition with $n$ copies of $1$.
The coinvariant ring $R_n$ presents 
the cohomology of complete flag variety in
type $\textrm{A}_{n-1}$.
The ideal $I_n$ and the ring $R_n$ 
have two other generalizations, 
both of which can be recovered 
as special cases of $I_{n, \lambda, s}$ and
$R_{n, \lambda, s}$.
\begin{itemize}
\item For a partition $\lambda$,
$I_{|\lambda|, \lambda, \infty}$ is the Tanisaki ideal $I_{\lambda}$ and 
$R_{|\lambda|, \lambda, \infty}$ is the Tanisaki quotient $R_{\lambda}$
studied by Tanisaki~\cite{Ta} and Garsia-Procesi~\cite{GP}.
The ring $R_\lambda$ presents the cohomology of the Springer fiber of $\lambda$.
\item For some
$k \leq n$ and $k \leq s < \infty$, the $I_{n, (1)^k, s}$
and $R_{n, (1)^k, s}$ correspond to the 
$I_{n,k,s}$ and $R_{n,k,s}$
introduced by Haglund-Rhoades-Shimozono~\cite{HRS}.
The ring $R_{n, k, s}$
gives a representation-theoretic model for the Haglund-Remmel-Wilson
Delta Conjecture~\cite{HRW}. 
Pawlowski and Rhoades proved that 
$R_{n, k, s}$ presents the cohomology of $n$-tuples $(\ell_1, ... , \ell_n)$ of lines in $\mathbb{C}^n$ such that the projection of $\ell_1 + .... + \ell_n$ onto $\mathbb{C}^k$ is surjective.
\end{itemize}

The theory of 
Gr\"{o}bner basis plays a crucial role in
the study of ideals in polynomial rings.
Most computation problems, such as determining
whether a given polynomial is in the ideal
or finding generators of intersections of ideals,
can be answered with Gr\"{o}bner bases.
Moreover, 
a Gr\"{o}bner basis can be used to find
a monomial basis of the quotient ring,
known as the standard monomial basis.
Haglund, Rhoades and Shimozono
found the reduced Gr\"{o}bner basis 
of $I_{n, k, s}$ to 
obtain a monomial basis of $R_{n, k, s}$.

In this paper, we find a Gr\"{o}bner basis of $I_{n, \lambda, s}$ with 
respect to the graded reverse lexicographical order.  
To the best of the author's knowledge, 
finding a Gr\"{o}bner basis for even Tanisaki ideals $I_{\lambda}$ remained open.
The form of our Gr\"{o}bner basis implies Griffin's monomial basis of $R_{n, \lambda, s}$ in~\cite{G} is the standard monomial basis.
One notable feature of our Gr\"{o}bner basis 
is that its polynomials have integer coefficients
with leading coefficient one.  
This is rarely the case,
even for ideals with nice generating sets 
and important ties to algebra, geometry and combinatorics.

The Gr\"{o}bner bases we construct
are not minimal
thus not reduced. 
However, since each polynomial has integer coefficients with leading coefficient $1$,
we can deduce
the polynomials in the reduced Gr\"{o}bner bases
also have integer coefficients.
We believe our work is a good starting point for addressing the following open problem:
\begin{prob}
Find the reduced Gr\"{o}bner bases of all $I_{n, \lambda, s}$.
\end{prob}

The rest of the paper is structured as follows. 
In \S\ref{S: Background}, 
we provide necessary background
on Griffin's ideals and Gr\"{o}bner bases.
We then deduce that 
a Gr\"{o}bner basis of $I_{n, \lambda, s}$
when $s < \infty$
can be obtained by adding 
$x_1^s, \cdots, x_n^s$ 
to a Gr\"{o}bner basis of $I_{n, \lambda, \infty}$.
In \S\ref{S: container},
we find a finite set of monomials $B_{n, \lambda}$ 
that generates the initial 
ideal of $I_{n, \lambda, \infty}$.
Our main tool is 
a combinatorial construction 
called container diagram
developed by Rhoades, Yu and Zhao~\cite{RYZ}.
Finally, in \S\ref{S: construction},
we recursively construct 
a polynomial in $I_{n, \lambda, \infty}$
with leading monomial $x^\alpha$ for
each $x^\alpha \in B_{n, \lambda}$.
Consequently, 
the polynomials we construct form a
Gr\"{o}bner basis of $I_{n, \lambda, \infty}$.

\section{Background}
\label{S: Background}

\subsection{Defining $I_{n, \lambda, s}$ and $R_{n, \lambda, s}$ }

Given $S \subseteq [n] := \{1, \cdots, n\}$,
let $e_d(S)$ be the \definition{elementary symmetric function} of degree $d$ with
variable set $\{x_i: i \in S\}$:
$$
e_d(S) := \sum_{\{i_1 < i_2 < \cdots < i_d\} \subseteq S} x_{i_1}x_{i_2} \cdots x_{i_d}.
$$
By convention, $e_{0}(S) = 1$ for any $S$ and
$e_d(S) = 0$ if $d > |S|$ or $d < 0$.
We recall one simple but useful identity:
\begin{equation}
\label{EQ: e}
e_d(S \sqcup \{ 1 \}) = e_d(S) + x_1 e_{d-1}(S),
\end{equation}
where $d$ is any integer and $S$ is a set not
containing $1$.

A \definition{partition} $\lambda$ is a weakly decreasing 
sequence of non-negative integers with
only finitely many positive entries.
When writing down a partition, 
we often ignore the trailing $0$s.
We denote its $i^{th}$ entry of $\lambda$ by $\lambda_i$.
Let $\ell(\lambda)$ and $|\lambda|$ denote the number of 
positive entries 
and the sum of entries in $\lambda$ respectively.
We say $\lambda$ is a partition of $n$ if 
$n = |\lambda|$.
The \definition{conjugate} of $\lambda$, denoted as $\lambda'$,
is a partition with $\lambda'_i = |j: \lambda_j \geq i|$.
Following Griffin's convention,
for $n \geq m$,
we define $p_m^n(\lambda) := \lambda'_{n - m + 1} + 
\lambda'_{n - m + 2} + \cdots$.

\begin{exa}
\label{E: partition}
If $\lambda = (3,2)$, we have $\lambda' = (2,2,1)$.
Say $n = 7$, then $p_7^7(\lambda)= 2 + 2 + 1 = 5$,
$p_6^7(\lambda) = 2 + 1 = 3$,
$p_5^7(\lambda) = 1$, and 
$p_m^7(\lambda) = 0$ if $m < 5$.
\end{exa}

Griffin's ideal $I_{n, \lambda, s}$ and 
Griffin's ring $R_{n, \lambda, s}$ can be defined as follows.

\begin{defn}\cite[Definition 3.0.1]{G}
\label{D: I}
Take a number $n$.
Let $\lambda$ be a partition with $1 \leq |\lambda| \leq n$.
If $s = \infty$, $I_{n,\lambda, s}$ is the ideal of 
$\mathbb{Q}[x_1, \cdots, x_n]$
generated by:
$$e_d(S) \textrm{ for all } 
S \subseteq [n], d > |S| - p_{|S|}^n(\lambda).$$
Now if $s$ is a number such that $s \geq \ell(\lambda)$.
Then $I_{n,\lambda, s}$ is the ideal generated by
the generators of $I_{n, \lambda, \infty}$ above
together with $x_1^s, \cdots, x_n^s$.

Finally, let $R_{n,\lambda,s} := \mathbb{Q}[x_1, \cdots, x_n]/I_{n,\lambda, s}$
for all $\ell(\lambda) \leq s \leq \infty$.
\end{defn}

\begin{exa}
Following Example~\ref{E: partition},
we let $n = 7$, $\lambda = (3,2)$.
Then $I_{n, \lambda, \infty}$ is generated by  $e_d(S)$ where
$d$ and $S \subseteq [7]$ can take the following values.
\begin{itemize}
\item If $|S| = 7$,
$d > |S| - 5 = 2$.
\item If $|S| = 6$,
$d > |S| - 3 = 3$.
\item If $|S| = 5$,
$d > |S| - 1 = 4$.
\item If $|S| < 5$,
$d > |S| - 0 = |S|$,
forcing $e_d(S)$ to vanish.
\end{itemize}
The ideal $I_{n, \lambda, 5}$ is generated by the generators above
together with $x_1^5, \cdots, x_7^5$.
\end{exa}

\begin{rem}
We may slightly extend the definition 
of $I_{n, \lambda, s}$ and $R_{n, \lambda, s}$.
When $|\lambda| > n$, we still define $I_{n, \lambda, s}$
and $R_{n, \lambda, s}$ in the same way as above. 
In this case, 
$p_n^n(\lambda) = |\lambda| > n$,
so $e_0([n])$ is a generator of $I_{n, \lambda, s}$.
Notice that $e_0([n]) = 1$, 
so $I_{n, \lambda, s}$ is just 
$\mathbb{Q}[x_1, \cdots, x_n]$ and 
$R_{n, \lambda, s}$ is the trivial ring. 
Our extension does not introduce anything interesting,
but makes some of our recursive arguments easier to state.  
\end{rem}

\subsection{Gr\"{o}bner basis}
We recall some fundamental background regarding 
the Gr\"{o}bner basis.
A \definition{monomial order} is a total order $<$ on the monomials
of $x_1, x_2, \cdots, x_n$
such that 
\begin{itemize}
\item For any monomial $m$, $1 \leq m$, and 
\item For any three monomials $m_1, m_2$ and $m_3$, $m_1 < m_2$ implies $m_1m_3 < m_2m_3$.
\end{itemize}

Let $<$ be a monomial order.
For $f$ in $\mathbb{Q}[x_1, \cdots, x_n]$,
its \definition{leading monomial} with respect to $<$, 
denote as $\leading_<(f)$,
is the largest monomial
with non-zero coefficient in $f$.
For an ideal $I$ of $\mathbb{Q}[x_1, \cdots, x_n]$,
its \definition{initial ideal},
denoted as $\leading_<(I)$,
is the ideal generated 
by $\leading_<(f)$ for $f \in I - \{0\}$.
Then $\{g_1, \dots, g_N\} \subseteq I$ is a  \definition{Gr\"{o}bner basis} of $I$ 
if $\langle \leading_<(g_1), \dots, \leading_<(g_N) \rangle = \leading_<(I)$.
This implies that 
$g_1, \dots, g_N$ generate $I$.
Furthermore,
we say $\{g_1, \dots, g_N\}$
is a \definition{reduced Gr\"{o}bner basis} if
\begin{itemize}
\item The coefficient of $\leading_<(g_i)$ in $g_i$ is 1, and
\item For $i \neq j$, $\leading_<(g_i)$ does not divide
any monomial with a non-zero coefficient in $g_j$.
\end{itemize}
Every ideal $I$ has exactly one reduced Gr\"{o}bner basis
with respect to a given monomial order. 

One application of Gr\"{o}bner bases is to 
find a monomial basis of the quotient ring. 
Consider an ideal $I$ and some monomial order $<$. 
A monomial $m$ is called a \definition{standard monomial}
if $m$ is not in $\leading_<(I)$.
It is known that 
$\{m + I: m \textrm{ is a standard monomial}\}$
forms a basis of $\mathbb{Q}[x_1, \cdots, x_n]/I$.
This basis is called the \definition{standard monomial basis}.
Notice that this basis is closed under
taking factor: 
If $m + I$ is in this basis and let $m'$
be a factor of $m$,
then $m'+I$ is also in this basis.
Determine whether a monomial is standard
can be a painful process since
one needs to determine whether a monomial
is in an ideal.
The Gr\"{o}bner basis can make this process
painless. 
Let $G$ be a Gr\"{o}bner basis
of $I$ with respect to $<$.
A monomial $m$ is a standard monomial if and only if
$\leading_<(g)$ does not divide $m$ for any $g \in G$.

One classical monomial order is the 
\definition{lexicographical order} $<_{lex}$.
We may define this order as 
$x_1^{a_1} x_2^{a_2} \cdots 
<_{lex} x_1^{b_1} x_2^{b_2} \cdots$
if there exists $i$ such that $a_j = b_j$
for $j > i$ and $a_j < b_j$.\footnote{To make our statements concise, we reverse the 
usual definition of monomial orders 
in this paper. 
Usually, the ``$j > i$'' in the definition
of lexicographical order 
is replaced by $j < i$.}
A Gr\"{o}bner basis can be hard to find
and its polynomials rarely have integer coefficients.
Nevertheless, the reduced Gr\"{o}bner bases of $I_n$, 
and more generally $I_{n,k,s}$, with
respect to the lexicographical order
are well-behaved classical polynomials.
It is a well-known result that 
$$\{h_d([n - d + 1]): d \in [n]\}, 
\textrm{ where } h_d([n - d + 1]) := \sum_{1 \leq i_1 \leq \cdots \leq i_d \leq n-d+1} x_{i_1} 
\cdots x_{i_d}$$ 
is 
the reduced Gr\"{o}bner basis of $I_n$ (see~\cite[Theorem 1.2.7]{Stu}
and~\cite[Sec 7.2]{Ber}).
Haglund, Rhoades and Shimozono extend
this result. 
In~\cite[Theorem 4.14]{HRS},
they show the reduced Gr\"{o}bner basis
of $I_{n,k,s}$ consists of certain \definition{key polynomials}
together with $x_1^s, \cdots, x_n^s$ when $k < s$.
A notable feature of this reduced Gr\"{o}bner basis is that
its polynomials all have positive integer coefficients. 

This paper considers a different monomial order.

\begin{defn}\cite[\S 2 Definition 6]{CLOS}
Define the \definition{graded reverse lexicographical order} $<_{grevlex}$
as $x_1^{a_1} x_2^{a_2} \cdots <_{grevlex} x_1^{b_1} x_2^{b_2} \cdots$
if either of the following holds
\begin{itemize}
\item The degree of $x_1^{a_1} x_2^{a_2} \cdots$ 
is less than
the degree of $x_1^{b_1} x_2^{b_2} \cdots$, or
\item They have the same degree and 
there exists $i$ such that $a_j = b_j$
for all $1 \leq j < i$ and $a_j > b_j$.
\end{itemize} 
\end{defn}

After switching to this order, 
the reduced Gr\"{o}bner bases mentioned above 
remain to be reduced Gr\"{o}bner bases.
\begin{itemize}
\item For each $h_{d}([n-d+1])$, its leading monomial 
is $x_{n - d + 1}^{d}$
with respect to either $<_{lex}$ or $<_{grevlex}$.
Consequently, $\{h_d([n - d + 1]): d \in [n]\}$ is 
still the reduced Gr\"{o}bner basis of $I_n$
with respect to $<_{grevlex}$.
\item 
Furthermore, 
the leading monomial of every key polynomial
is the same with respect to $<_{lex}$ or $<_{grevlex}$.
Thus, the reduced Gr\"{o}bner basis found by Haglund, Rhoades and Shimozono
is also the reduced Gr\"{o}bner basis of $I_{n,k,s}$
with respect to $<_{grevlex}$.
\end{itemize}

We conjecture this pattern holds in general:
\begin{conj}
The reduced Gr\"{o}bner basis of $I_{n,\lambda,s}$
is the same for $<_{grevlex}$ or $<_{lex}$.
\end{conj}
This conjecture has been computer checked for 
all $n \leq 7$ and $s \leq 7$.

Hereafter, 
we use
$<$ to denote the
graded reverse lexicographical order,
in place of $<_{grevlex}$.

\subsection{Griffin's monomial basis of $R_{n,\lambda,s}$}

Before constructing the Gr\"{o}bner basis,
we need to know what monomials are in 
the initial ideal $\leading_{<}(I_{n, \lambda,s})$.
We may predict the structure of 
$\leading_{<}(I_{n, \lambda,s})$
from
Griffin's study of $R_{n, \lambda, s}$.
First,  recall a few definitions in Griffin's thesis.

A \definition{weak composition} of length $n$
is a sequence of $n$ non-negative integers. 
If $\alpha$ is a weak composition, we use $\alpha_i$
to denote its $i^{th}$ entry.
Given weak compositions 
$\alpha, \beta$. 
we say $\gamma$ is 
a \definition{shuffle} of $\alpha$ and $\beta$
if $\gamma$ is obtained 
by inserting entries of $\alpha$
into $\beta$ while preserving the relative order.
A shuffle of more than two weak compositions can be defined similarly (or recursively). 
By convention, the unique shuffle of a single
weak composition is just itself.

Let $n$ be a positive number. 
The \definition{staircase} of length $n$
is the weak composition 
$\rho^{(n)} := (n-1, n-2, \dots, 1,0)$.\footnote{Notice that our convention of a staircase
is the reverse of the Griffin's convention.}
Take a partition $\lambda$ with $1 \leq |\lambda| \leq n$
and a number $\ell(\lambda) \leq s < \infty$.
We say a weak composition is a  
\definition{$(n, \lambda, s)$-staircase}
if it is a shuffle of the following weak compositions:
$$
\rho^{\lambda'_1}, \cdots, \rho^{\lambda'_{\lambda_1}}, (s-1)^{n - |\lambda|},
$$
where the last weak composition consists 
of $n - |\lambda|$ copies of $s-1$.
Let $\C_{n, \lambda,s}$ be the set of weak compositions of length $n$ 
that are entry-wise less than or equal to
some $(n, \lambda, s)$-staircase.
Finally, let $\C_{n, \lambda,\infty} := \bigcup_{s \geq \ell(\lambda)} \C_{n, \lambda,s}$.

Let $\alpha$ be a 
weak compositions of length $n$. 
We use $x^\alpha$ to denote the monomial
$x_1^{\alpha_1} \cdots x_n^{\alpha_n}$.
Then define 
$\A_{n,\lambda, s}:= 
\{x^\alpha: \alpha \in \C_{n,\lambda, s}\}$.
These monomials form Griffin's monomial basis.

\begin{thm}\cite[Theorem 3.1.17]{G}
The quotient ring $R_{n, \lambda, s}$
has basis $\{m + I_{n, \lambda, s}: 
m \in \A_{n, \lambda, s}\}$.
\end{thm}

\begin{rem}
Consider $(n, \lambda, s)$ where $|\lambda| > n$.
Recall that we generalize $I_{n, \lambda, s}$ and $R_{n, \lambda, s}$
as $\mathbb{Q}[x_1, \cdots, x_n]$
and the trivial ring when $|\lambda| > n$.
We may set $\A_{n, \lambda, s} = \C_{n, \lambda, s} = \emptyset$.
Then Griffin's theorem still holds in this case: 
The trivial ring has the empty set as its basis. 
Again, our generalization is not introducing anything interesting,
but will make our later work easier to state.  
\end{rem}

Notice that Griffin's monomial basis 
is closed under taking factor.
This is a property shared by 
all standard monomial bases. 
It turns out that Griffin's basis
is the standard monomial basis
of $R_{n, \lambda, s}$
with respect to the graded reverse 
lexicographical order.
Let $\A_{n,\lambda,s}^C$ 
denote the set of monomials of 
$x_1, \cdots, x_n$ that are not in 
$\A_{n, \lambda, s}$.
We deduce the following: 

\begin{lemma}
\label{L: main}
Take $(n,\lambda)$
with $|\lambda| \leq n$.
If we can find a finite 
set of polynomials $G_{n, \lambda} \subseteq I_{n, \lambda, \infty}$
such that 
$$\langle \{\leading_<(F): F \in G_{n, \lambda} \}\rangle \: = \: \langle A^C_{n, \lambda, \infty} \rangle,$$
then $G_{n, \lambda}$
is a Gr\"{o}bner basis
of $I_{n, \lambda, \infty}$.
Moreover, $A_{n, \lambda, \infty}$
is the set of standard monomials of $I_{n, \lambda, \infty}$.

Now take any $s$ such that
$\ell(\lambda) \leq s < \infty$.
Then $G_{n, \lambda} \cup \{ x_1^s, \cdots, x_n^s\}$
is a Gr\"{o}bner basis
of $I_{n, \lambda, s}$.
Moreover, $A_{n, \lambda, s}$
is the set of standard monomials of $I_{n, \lambda, s}$.

\end{lemma}

\begin{proof}
The existence of $G_{n, \lambda}$ implies that
$\langle \A^C_{n,\lambda, \infty} \rangle \subseteq
\leading_<(I_{n, \lambda, \infty})$.
Thus, the set of standard monomials
of $I_{n, \lambda, \infty}$
is a subset of $\A_{n, \lambda, \infty}$.
Recall the standard monomials descends to a basis 
of $R_{n,\lambda, \infty}$.
A subset of Griffin's monomial basis is a basis, 
so this ``subset'' has to be the whole set. 
In other words, monomials in 
$A_{n, \lambda, \infty}$ are all standard monomials.
Since any monomial in $A_{n, \lambda, \infty}$
is not in $\leading_{<}(I_{n, \lambda, \infty})$,
we know $\langle \A^C_{n,\lambda, \infty} \rangle =
\leading_<(I_{n, \lambda, \infty})$.
Thus, $G_{n, \lambda}$
is a Gr\"{o}bner basis.

Now take $\ell(\lambda) < s < \infty$.
We can write $\A^C_{n, \lambda, s}$
as a union of two sets:
$$
\A^C_{n, \lambda, \infty}
\quad \bigcup \quad
\{\textrm{monomials divisible by some $x_i^s$}\} .
$$

The first set is a subset of $\leading_<(I_{n, \lambda, \infty})$,
which is a subset of $\leading_<(I_{n, \lambda, s})$.
The second set of monomials
is a subset of $\leading_<(I_{n, \lambda, s})$ since $x_1^s, \cdots, x_n^s$ are generators of $I_{n, \lambda, s}$.
Again, we have $\langle \A^C_{n,\lambda, s} \rangle \subseteq
\leading_<(I_{n, \lambda, s})$.
We perform the argument
in the previous paragraph.
Consequently, 
we know $A_{n,\lambda, s}$
is the set of standard monomials 
of $I_{n, \lambda, s}$
and
$\langle A^C_{n, \lambda, s}\rangle = \leading_<(I_{n, \lambda, s})$.
Finally, since 
$G_{n, \lambda} \subseteq I_{n, \lambda, \infty} \subseteq I_{n, \lambda, s}$,
the set $G_{n, \lambda} \cup \{x_1^s, \cdots, x_n^s\}$
is a Gr\"{o}bner basis
of $I_{n, \lambda, s}$.
\end{proof}

The rest of this paper aims to find the $G_{n, \lambda}$ in Lemma~\ref{L: main}.

\section{Finding generators of $\leading_<(I_{n, \lambda, \infty})$ via container diagrams}

We will find a finite set
of monomials $B_{n, \lambda}$
that generates the same
ideal as $A_{n,\lambda, \infty}^C$.
First, we recall some combinatorics construction developed 
by Rhoades, Yu and Zhao~\cite{RYZ}.

\label{S: container}
\subsection{Container diagram and coinversion code}
The \definition{Young diagram} of a partition $\lambda$
is a diagram consisting of boxes: 
row $i$ of the diagram has a box at column 
$1, 2, \cdots, \lambda_i$.
We use the convention that row $1$ is the topmost row
and column $1$ is the leftmost column.
Given $n$, $\lambda$, 
Rhoades, Yu and Zhao define the \definition{container
diagram}:
This is the Young diagram of $\lambda'$ together with numbers of $[n]$
filled.
Each number can either be filled in a box
of the Young diagram
or just float somewhere above the Young diagram
in a column.
Numbers that are not in boxes are called 
the \definition{floating numbers}.
The filling should satisfy all of the following. 
\begin{enumerate}
\item Every number in $[n]$ appears exactly once.
\item Every box can contain at most one number.
\item An empty box cannot have any number above it in its column.
\item Numbers are decreasing from top to bottom
in each column.
\end{enumerate}

\begin{exa}
\label{E: container}
Let $n = 11$ and $\lambda = (3,2,2,1)$.
Then following would be a container diagram
of $(n, \lambda)$:
$$
\sigma \: = \: \raisebox{1cm}{
\begin{ytableau}
\none & \none & \none[10] & \none & \none \cr
\none[11] & \none & \none[7] & \none & \none[5] \cr
4 &  & 6 & 9 \cr
3 & 8 & 1  \cr
2 \cr
\end{ytableau}}
$$
\end{exa}

\begin{rem}
Our convention of the container diagram is slightly different from 
the convention in~\cite{RYZ}.
First, we allow empty boxes to appear in the container diagram. 
Second, we do not require $n \geq |\lambda|$.
\end{rem}

Rhoades, Yu and Zhao 
use container diagrams to represent 
\definition{ordered set partitions},
an ordered infinite sequence of pair-wise disjoint
sets such that the union is $[n]$,
where each part is allowed to be $\emptyset$. 
Naturally, each container diagram of 
$(n, \lambda)$ corresponds to 
an ordered set partition of $[n]$.
For instance, the diagram in the previous
example represents:
$$(\{2,3,4,11\}, \{8\}, \{1,6,7,10\}, \{9\}, \{5\}, \emptyset, \cdots)$$
where all trailing sets are empty.
For this reason, we denote the set of 
container diagrams for $(n, \lambda)$ as $\OP_{n, \lambda}$.

Rhoades Yu and Zhao define the following
map from $\OP_{n, \lambda}$ to $\mathbb{Z}_{\geq 0}^n$,
the set of weak compositions with length $n$.
Notice that they only 
define the map on container diagrams whose
boxes are all filled. 
We can simply extend their definition by
imagining $\infty$ is filled in all empty boxes. 
\begin{defn}\cite{RYZ}
Let $\sigma$ be an container diagram of $(n, \lambda)$.
We define $\code(\sigma)$ 
as a weak compositions of length $n$.
Its $i\textsuperscript{th}$ entry 
can be found as follows. 
\begin{itemize}
\item Say the number $i$ lives in a box in $\sigma$.
Assume $i$ is in row $r$.
Consider the boxes that are in row $r$
to the right of $i$ or in row $r+1$ to the left of $i$.
Then $\code(\sigma)_i$ is the number
of boxes in this region 
that contain a number larger than $i$.

\item Otherwise, $i$ is floating.
Assume $i$ is in column $c$. 
Then $\code(\sigma)_i$ is $c-1$ plus
the number of boxes in row 1 to the right of column $c$ that contain a number larger than $i$.
\end{itemize}
Remember that when computing $\code(\sigma)$,
we imagine all empty boxes of the Yound diagram
contains $\infty$.
We call $\code(\sigma)$ the \definition{coinversion code}
of $\sigma$.
\end{defn}

If $\sigma$ is the container diagram in Example~\ref{E: container},
then 
$$
\code(\sigma) = (1,0,1,3,4,2,3,0,0,2,1).
$$

It turns out that $\code$ is a bijection
from $\OP_{n, \lambda}$ to $\mathbb{Z}_{\geq 0}^n$.
An inverse of the map is given 
in the proof of~\cite[Theorem 3.6]{RYZ}.
We briefly describe this inverse $\code^{-1}$. 
Given weak composition $\alpha$ with length $n$,
we compute $\sigma = \code^{-1}(\alpha) \in \OP_{n,\lambda}$
via the following algorithm. 
We start from the Young diagram of $\lambda'$ and insert
$1, 2, \dots, n$ iteratively.
Suppose we have inserted $1, \dots, t-1$
and we want to insert $t$.
For each column, 
$t$ can only go to the bottom-most available spot within the column. 
We determine which column $t$ should go to
based on the value of $\alpha_t$.
First, arrange the positive integers $1, 2, 3, \cdots$ 
in the following way.
\begin{itemize}
\item Suppose column $i$ has an empty box but
column $j$ does not. Then $i$ comes before $j$.
\item  Suppose column $i$ and column $j$
have no empty boxes. 
Then $i$ comes before $j$
if $i < j$.
\item Suppose column $i$ and column $j$
both have empty boxes.
Say the bottom-most empty box in column $i$ (resp. $j$)
is at row $r_i$ (resp. $r_j$).
Then $i$ comes before $j$
if $r_i > r_j$ or $r_i = r_j$ and $j < i$.
\end{itemize}
Let $c_1, c_2, c_3, \cdots$ be the resulting rearrangement
of $1, 2, 3, \cdots$.
Readers may check that if we 
put $t$ at column $c_p$, 
the $t^\textsuperscript{th}$ entry in the 
coinversion code of the resulting container diagram
must be $p - 1$.
Thus, we insert $t$ into column
$c_{\alpha_t + 1}$.
We obtain $\code^{-1}(\alpha)$ after inserting
$1, 2, \dots, n$.

\begin{exa}
We present one step in computing $\code^{-1}((1,0,1,3,4,2,3,0,0,2,1))
$
with $n = 11$ and $\lambda = (3,2,2,1)$.
After inserting $1$ and $2$ to the empty Young diagram of $\lambda'$,
we have 
$$
\begin{ytableau}
\: & \: & \: & \: \cr
3 & \: & 1\cr
2 \cr
\end{ytableau}
$$
Now we insert $4$.
By the algorithm, we may 
arrange the positive numbers into 
$(c_1, c_2, \cdots) = (2,4,3,1,5,6,7,\cdots)$.
Putting the number $4$
at the bottom of column 
$c_t$ will yield a container diagram
whose coinversion code has $t+1$ as 
the $4\textsuperscript{th}$
entry.
Since we want to have $3$
as the $4\textsuperscript{th}$
entry,
we put $4$ at the bottom of
column $c_{3 + 1} = 1$.
After inserting all numbers in $[11]$,
we obtain the container diagram in Example~\ref{E: container}.

\end{exa}

\begin{rem}
\label{R: Insert 0}
We make one simple observation about the
insertion that will be used later.
Suppose we are computing $\code^{-1}(\alpha)$
and we are about insert $t$.
If $\alpha_t = 0$ and there is at least one
empty box in the diagram, 
we will put $t$ in a box. 
\end{rem}

One use of the container diagram and the coinversion code is to identify
whether a weak composition lies in $\C_{n, \lambda, \infty}$.
Rhoades, Yu and Zhao describe its image
under $\code^{-1}(\cdot)$.
\begin{thm}\cite[Theorem 3.6]{RYZ}
\label{T: container no empty}
Consider $(n, \lambda)$.
The map $\code(\cdot)$ and $\code^{-1}(\cdot)$ restrict to a bijections
between $\C_{n, \lambda, \infty} \subseteq \mathbb{Z}_{\geq 0}^n$ and 
$$\{ \sigma \in \OP_{n, \lambda}: 
\sigma \textrm{ has no empty boxes in the Young diagram}\} \subseteq \OP_{n, \lambda}.$$
\end{thm}

Technically, Rhoades, Yu
and Zhao only prove this theorem when $n \geq |\lambda|$.
When $n < |\lambda|$, notice that every container diagram 
in $\OP_{n, \lambda}$ must have a empty box.
Recall that we set $\C_{n, \lambda,\infty} = \emptyset$,
so this theorem holds trivially. 

\subsection{Finding a finite generating set
for $\langle \A_{n, \lambda, \infty}^C \rangle$}

In this subsection, 
we use container diagrams to describe a set of monomials $\B_{n, \lambda}$
that can generate the same ideal as $\A_{n, \lambda, \infty}^C$.
Notice that this process is immediate when 
$|\lambda| > n$,
in which case $\A_{n, \lambda, \infty} = \emptyset$.
We only need to worry about the case when
$|\lambda|\leq n$.
However, we need to make use of a recursive 
structure of $\B_{n, \lambda}$ later,
so we choose to focus on $(n, \lambda)$
such that $1 \leq |\lambda| \leq n+1$.
Now, fix such $(n, \lambda)$ throughout
this subsection.

By Theorem~\ref{T: container no empty},
we know $\A_{n, \lambda, \infty}^C$ can be written as
$$
\{ x^{\code(\sigma)}: \sigma \in \OP_{n, \lambda, s} \textrm{ with an empty box}\}.
$$

Let $\C^C_{n, \lambda,\infty}$ be the set of weak compositions
that are not in $\C_{n, \lambda,\infty}$.
We describe a set 
$\D_{n, \lambda} \subseteq \C^C_{n, \lambda,\infty}$,
Then we check
for any $\alpha$
in $\C^C_{n, \lambda,\infty}$,
we can find $\beta \in \D_{n, \lambda}$
such that $\beta \leq_{e} \alpha$.
Here, $\leq_e$ denotes entry-wise less than or equal to.
Finally, we set
$$\B_{n, \lambda}:=\{x^\beta: \beta \in \D_{n, \lambda}\}.$$
Clearly, $\B_{n,\lambda}$ generates the same ideal 
as $\A_{n, \lambda, \infty}^C$.
Now,
we define $\D_{n, \lambda}$
and check it has the desired properties. 

\begin{defn}
\label{D: D}
Let $\D_{n,\lambda}$
be the set of $\code(\sigma)$ 
where $\sigma$ ranges over elements in $\OP_{n, \lambda}$
satisfying all of the following three conditions. 
\begin{itemize}
\item Condition 1: The Young diagram in $\sigma$
has exactly one empty box. 
\item Condition 2: The filling of the Young diagram is 
decreasing in each row from left to right.
Moreover, the empty box in the Young diagram is at row 1 column 1.
\item Condition 3: If a number $i$ is floating, 
then every column on its left has a box that is empty
of contains a number larger than $i$.
\end{itemize}
\end{defn}

\begin{exa}
\label{E: 3 conditions}
Consider the $(n, \lambda)$ and the
container diagram $\sigma$ in Example~\ref{E: container}.
Clearly, $\sigma$ satisfies condition 1.
However, $\sigma$ does not satisfy condition 2,
since the only box is not at column $1$. 
Moreover, the rows are not decreasing from left to right.
It does not satisfy condition 3 either,
since $5$ is floating but the $4$ in column $1$ row $1$
is less than $5$.
Here is container diagram $\gamma$ that 
satisfies all three conditions for the same $(n, \lambda)$:
$$
\gamma = 
\raisebox{1cm}{
\begin{ytableau}
\none & \none[11] & \none & \none & \none \cr
\none & \none[10] & \none[7] & \none[5] & \none \cr
  & 9 & 6 & 4 \cr
8 & 3 & 1  \cr
2 \cr
\end{ytableau}}
$$

Thus, $\code(\gamma) = (1,0,0,1,3,1,2,0,0,1,1)$ is in $\D_{n, \lambda}$.
\end{exa}

Then we check $\D_{n,\lambda}$ has the 
desired property.
\begin{prop}
\label{P: D}
For any $\alpha \in \C^C_{n, \lambda, \infty}$, 
we can find $\beta \in \D_{n,\lambda}$
with $\beta \leq_e \alpha$.
Thus, $\langle B_{n, \lambda} \rangle = \langle A_{n, \lambda, \infty}^C\rangle$.
\end{prop}

\begin{exa}
Consider the $(n, \lambda)$ and the
container diagram $\sigma$ in Example~\ref{E: container}.
Let $\alpha = \code(\sigma)$.
Since $\sigma$ contains an empty box,
we know $\alpha \in \C^C_{n, \lambda, \infty}$.
Then $\beta$ can be $\code(\gamma)$ where $\gamma$ is
the container diagram in Example~\ref{E: 3 conditions}.
We have $\beta \in \D_{n, \lambda}$
and $\beta \leq_e \alpha$.
\end{exa}

\begin{rem}
It is important to note that
$\D_{n,\lambda}$
is not the smallest subset that
satisfies Proposition~\ref{P: D}.
In other words, 
two distinct elements $\alpha$ and $\beta$
in $\D_{n,\lambda}$ 
may satisfy $\alpha \leq_e \beta$.
Readers may check 
when $n = 4$, $\lambda = (3,1)$,
$\D_{n, \lambda}$
contains both $(0,0,0,1)$
and $(1,0,0,1)$.
Thus, 
if our Gr\"{o}bner basis
contains a polynomial of leading monomial $x^\alpha$
for each $\alpha \in \D_{n, \lambda}$,
it cannot be the reduced Gr\"{o}bner basis.
It would be a good problem to find a new definition
of $\D_{n,\lambda}$
such that Proposition~\ref{P: D} is satisfied
and $\D_{n,\lambda}$ is minimal. 
\end{rem}

Now we prove Proposition~\ref{P: D}.
The idea is to turn the proposition into an equivalent
statement regarding $\OP_{n,\lambda}$.
\begin{proof}[Proof of Proposition~\ref{P: D}]
Recall that $\code(\cdot)$ is a bijection from 
$\OP_{n, \lambda}$ to $\mathbb{Z}_{\geq 0}^n$.
By Theorem~\ref{T: container no empty},
it is enough to show the following claim:

\noindent\textbf{Claim:} 
For any $\sigma \in \OP_{n, \lambda}$
with at least one empty box,
we can find $\gamma \in \OP_{n, \lambda}$
satisfying the three conditions in the Definition~\ref{D: D}
such that 
$\code(\gamma) \leq_e \code(\sigma)$.

Our approach to show the claim is straightforward, 
consisting of the following three lemmas. 
It is clear that these three lemmas
would imply our claim,
hence finishing the proof. 
\end{proof}

\begin{lemma}
\label{L: Step 1}
 For any $\sigma \in \OP_{n, \lambda}$
with at least one empty box,
we can find $\gamma \in \OP_{n, \lambda}$
satisfying Condition 1
such that $\code(\gamma) \leq_e \code(\sigma)$.    
\end{lemma}

\begin{lemma}
\label{L: Step 2}
For any $\sigma \in \OP_{n, \lambda}$
satisfying condition 1,
we can find $\gamma \in \OP_{n, \lambda}$
satisfying condition 1 and 2
such that $\code(\gamma) \leq_e \code(\sigma)$.
\end{lemma}

\begin{lemma}
\label{L: Step 3}
For any $\sigma \in \OP_{n, \lambda}$
satisfying condition 1 and 2,
we can find $\gamma \in \OP_{n, \lambda}$
satisfying condition 1, 2 and 3 
such that $\code(\gamma) \leq_e \code(\sigma)$.
\end{lemma}

Now we prove these three Lemmas.
The proofs are all constructive: 
we give an algorithm to construct the $\gamma$
from $\sigma$ in each proof. 

\begin{proof}[Proof of Lemma~\ref{L: Step 1}]
If $\sigma$ has exactly one empty box, 
we are done by setting $\gamma = \sigma$.
Now assume $\sigma$ has more than one box. 
Let $\alpha = \code(\sigma)$
and consider the insertion algorithm that computes
$\code^{-1}(\alpha)$.
We focus on the quantity: 
$$
\textrm{number of empty boxes} \quad - \quad \textrm{number of elements in $[n]$
that have not been inserted}.
$$
Recall that we assume $|\lambda| \leq n+1$
throughout this subsection.
Thus, this quantity was $|\lambda| - n \leq 1$
when the algorithms starts,
and becomes larger than $1$
when the algorithm ends.
Moreover, each iteration can only 
increase this quantity by $0$ or $1$.
Thus, we can find $N < n$ such that right before inserting $N$,
this quantity is $1$.
At this moment, 
there are $n - N + 1$ numbers waiting
to be inserted. 
Thus,
there are $n - N + 2$
empty boxes in diagram.
We obtain weak composition $\beta$ 
from $\alpha$
by setting $\beta_i = \alpha$ if $i < N$
and $\beta_i = 0$ if $i \geq N$.
Clearly, $\beta \leq_e \alpha$.
When we compute $\code^{-1}(\beta)$,
the first $N - 1$ iterations behave the same 
as computing $\code^{-1}(\alpha)$.
By Remark~\ref{R: Insert 0},
the algorithm will insert $N, N+1, \cdots, n$
into boxes, leaving exactly one empty box at the end.
Thus, $\code^{-1}(\beta)$ is the $\gamma$ we want.
\end{proof}

\begin{proof}[Proof of Lemma~\ref{L: Step 2}]
In this proof, we imagine the empty spaces of a container diagram are filled with $\infty$.
Define the row inversion of a container diagram to be a pair of boxes on the same row
with the box on the left containing 
a larger number than the box on the right.
With this convention and definition, condition 2 can be rephrased as: ``There are no row inversions.''

Now suppose $\sigma \in \OP_{n,\lambda}$
satisfies condition 1 and 
has at least one row inversion.
We just need to find $\sigma' \in \OP_{n, \lambda}$ satisfying
condition 1 and having less row inversions
than $\sigma$.
Moreover, we need to make sure
$\code(\sigma') \leq_e \code(\sigma)$.

Look at the bottom-most row of boxes in $\sigma$ where a row inversion appears. 
We can find two adjacent boxes in this row
containing $l_0, r_0$ with $l_0 < r_0$
and $l_0$ is on the left. 
(The name $l$ stands for ``left''
and $r$ stands for ``right''). 
Let $l_1 < l_2 < \cdots < l_m$ be the numbers 
in boxes above $l_0$.
Let $l_{-1}$ be the number immediately underneath $l_0$, 
if it exists. 
Define $r_{-1}, r_1, r_2, \cdots$ in the same way. 
These two columns of the boxes in $\sigma$
looks like: 

$$
\begin{ytableau}
l_m & r_m \cr
\none[\vdots] & \none[\vdots]\cr
l_1 & r_1 \cr
l_0 & r_0 \cr
l_{-1} & r_{-1} \cr
\none[\vdots] & \none[\vdots]\cr
\end{ytableau}
$$

Now we obtain $\sigma'$ by doing the following
operations to $\sigma$.
Find the largest $i \in [m]$
such that $l_{j} < r_j$
for $j = 0, 1, \dots, i$.
Then for $j = 0, 1, \dots, i$,
we swap $l_j$ and $r_j$.
If $i = m$,
we also look at the numbers that 
were floating above $l_m$ in $\sigma$.
We find the floating numbers that are 
less than $r_m$
and move them to float in the column
on the right.

First, we check $\sigma'$ is a valid 
container diagram. 
We just need to make sure the columns are 
strictly decreasing from top to bottom. 
\begin{itemize}
\item If $l_{-1}$ exists, 
we check it is smaller than the number above it in $\sigma'$,
which is $r_0$:
$r_0 > l_0 > l_{-1}$.
\item If $r_{-1}$ exists, since we pick the lowest row
in $\sigma$ with a row inversion, we know
$l_0 > l_{-1} > r_{-1}$.
Thus, $r_{-1}$ is smaller than the number above it in $\sigma'$.
\item If $i < m$,
we need to check $r_i$ and $l_i$
are smaller than the numbers above
them in $\sigma'$:
$l_{i+1} > r_{i+1} > r_i$ by the maximality of $i$
and $r_{i+1} > r_i > l_i$.
\item If $i = m$,
we need to check $r_m$ and $l_m$
are smaller than the numbers floating
above them in $\sigma'$.
If a number is floating above $r_m$,
by our construction,
it is larger than $r_m$.
If a number is floating above $l_m$,
there are two possibilities:
it was floating above $r_m$ or $l_m$
in $\sigma$.
In either case, we know it is larger
than $l_m$.
\end{itemize}

Finally, $\sigma'$ clearly has less row inversions than $\sigma$.
It remains to check $\code(\sigma') \leq_e \code(\sigma)$.
A routine analysis would yield:  $\code(\sigma')_j = \code(\sigma)_j - 1$
if $j = l_0, \cdots, l_i$, 
and $\code(\sigma')_j = \code(\sigma)_j$ otherwise. 
\end{proof}

\begin{proof}[Proof of Lemma~\ref{L: Step 2}]
Suppose the floating number $i$ in $\sigma$
does not satisfy condition $3$.
In other words, a column on the left of $j$ does
not have boxes that are empty 
or containing a number larger than $i$.
We simply ask $i$ to float in that column.
This operation clearly reduces the
$i\textsuperscript{th}$ entry 
of the coinversion code
and preserves condition 1 and 2. 
We may eliminate all violations of condition
3 by repeatedly applying this operation. 
\end{proof}

\subsection{Recursive description of $\D_{n, \lambda}$}
We end this section by observing a recursive
structure of $\D_{n, \lambda}$.
We need the following  
notation from Griffin's thesis: 
\begin{defn}
Let $\lambda$ be a partition.
We define $\lambda^{(j)}$ as the partition
such that the conjugate of $\lambda^{(j)}$ 
is obtained from $\lambda'$ 
by decreasing its $j^\textsuperscript{th}$
entry by $1$. 
By convention, we let $\lambda^{(0)} = \lambda$.
\end{defn}
Notice that $\lambda^{(j)}$ is not well-defined for all $j$. 
Clearly, $\lambda^{(j)}$
is well-defined if and only if the rightmost
cell in row $j$ of
the Young diagram of $\lambda'$ is the bottommost cell in its column. 

\begin{exa}
Let $\lambda = (3,2)$.
Then 
$$
\lambda^{(0)} = (3,2), \quad 
\lambda^{(2)} = (3,1), \quad
\lambda^{(3)} = (2,2), 
$$    
but $\lambda^{(1)}$
is not defined. 
\end{exa}

\begin{prop}
\label{P: recursive D}
If $1 \leq |\lambda| \leq n + 1$ and $n > 1$,
we can write $\D_{n,\lambda}$
recursively as
$$
\D_{n, \lambda} = \bigsqcup_{j} 
(\lambda_{j+1}' \circ \D_{n-1,\lambda^{(j)}}),
$$
where $j$ ranges over $j$
such that $\lambda^{(j)}$ is defined,
but $j$ cannot be $0$ if $|\lambda| = n+1$.
Here, $\lambda_{j+1}' \circ $
is the operator that prepends $\lambda_{j+1}'$
to a weak composition.

For the base case,
we have $D_{1, (1)} = \{(1)\}$
and $D_{1, (2)} = D_{1, (1,1)} = \{(0)\}$
\end{prop}
\begin{proof}
First, assume $n > 1$. 
Recall that $\D_{n, \lambda}$
is the set of coinversion codes of $\sigma \in \OP_{n, \lambda}$
where $\sigma$ satisfies the three conditions in Definition~\ref{D: D}.
Consider such a $\sigma$.

Suppose $1$ is not floating in $\sigma$, by condition 2,
there are no boxes to its right or under it. 
Thus, $1$ is in the rightmost box of row $j$ 
for some $j$ that $\lambda^{(j)}$ is defined.
Let $\sigma'$ be the container diagram obtained by 
ignore this box and decrease all other numbers by 1. 
Clearly, $\sigma'$ is in $\OP_{n, \lambda^{(j)}}$
and satisfies the three conditions. 
In addition,
$$\code(\sigma) = \code(\sigma)_1 \circ \code(\sigma').$$
Since $1$ is on the right end of row $j$,
we know $\code(\sigma)_1$ is just
the number of boxes in row $j+1$ of $\sigma$,
which is $\lambda_{j+1}'$.

Now suppose $1$ is floating in $\sigma$.
This can happen only when $|\lambda| \leq n$.
Condition $3$ forces the $1$ 
to float in column $\ell(\lambda) + 1$.
Similarly, we may let $\sigma'$ be the container diagram obtained by 
ignore $1$ in $\sigma$ and decrease all other numbers by 1. 
Then $\sigma' \in \OP_{n, \lambda^{(0)}}$ satisfying 
the three conditions and $\code(\sigma) 
= \code(\sigma)_1 \circ \code(\sigma')$.
We see $\code(\sigma)_1 = \ell(\lambda) = \lambda_1'$,
which proves the recurrence relation.
The base cases are immediate. 
\end{proof}

\section{Recursive construction of the Gr\"{o}bner basis}
\label{S: construction}
Consider $n$ and $\lambda$ with $1 \leq |\lambda| \leq n + 1$.
In this section, 
we construct $G_{n, \lambda} \subseteq 
I_{n,\lambda, \infty}$ such that
$\{\leading_<(F): F \in G_{n, \lambda}\}
= B_{n, \lambda}$.
Define $J_{n, \lambda}$ as an ideal of $\mathbb{Z}[x_1, \cdots, x_n]$ with the same generators
as $I_{n, \lambda,\infty}$ in Definition~\ref{D: I}.
In fact, our $G_{n, \lambda}$
will be a subset of $J_{n, \lambda}$.
When $n = 1$,
we simply let $G_{n,(1)} = \{e_1([1])\} = \{x_1\}$
and $G_{n, (2)} = G_{n,(1,1)} = \{e_0([1])\} = \{1\}$.
For $n > 1$, 
our recursive approach can be summarized as follows:

\begin{rem}
\label{R: approach}
Let $\alpha = (\alpha_1, \cdots, \alpha_n)$ 
be a weak composition of length $n$.
Let $J$ be an ideal of $\mathbb{Z}[x_2, \cdots, x_n]$
generated by $f_1, \dots, f_m$.
Suppose we have a polynomial $f \in J$
with $\leading_<(f) = x_2^{\alpha_2} \cdots x_n^{\alpha_n}$.
We may write $f$ as 
$f_1 g_1 + \cdots + f_n g_n$
for some $g_1, \cdots, g_n \in \mathbb{Z}[x_2, \cdots, x_n]$.

Now consider homogeneous
polynomials $F_1, \cdots, F_m \in \mathbb{Z}[x_1, \cdots, x_n]$ such that each $F_i - x_1^{\alpha_1} f_i$ 
is divisible by $x_1^{\alpha_1 + 1}$.
In other words, $F_i  = x_1^{\alpha_1} f_i + x_1^{\alpha_1 + 1} \tilde{F}_i$ for some polynomial $\tilde{F}_i$.
Then $F = F_1 g_1 + \cdots + F_m g_m$
is a homogeneous polynomial.
Moreover, 
$$
F = \sum_{i = 1}^n F_ig_i = 
\sum_{i = 1}^n x_1^{\alpha_1} f_i g_i + \sum_{i = 1}^n x_1^{\alpha_1 + 1} \tilde{F}_i g_i
=
x_1^{\alpha_1} \sum_{i = 1}^n  f_i g_i + 
x_1^{\alpha_1 + 1}\sum_{i = 1}^n  \tilde{F}_i g_i
=
x_1^{\alpha_1} f + 
x_1^{\alpha_1 + 1}\sum_{i = 1}^n  \tilde{F}_i g_i.
$$
Thus, $\leading_{<}(F) = x^{\alpha_1}\leading_{<}(f) = x^\alpha$.

If the leading monomial in $f$
has coefficient $1$,
then so is the leading monomial in $F$.
\end{rem}

Now fix $n > 1$ and $1 \leq |\lambda| \leq n+1$.
Take any $j$ such that $\lambda^{(j)}$ is defined.
Let $J'_{n, \lambda^{(j)}}$
be the ideal of 
$\mathbb{Z}[x_2 \dots, x_n]$ obtained from $J_{n-1,\lambda^{(j)}}$ by shifting each $x_i$ to $x_{i+1}$.
In other words, 
$J'_{n, \lambda^{(j)}}$ is generated by 
$e_d(S)$ where $S \subseteq \{2, \cdots, n\}$
and $d > |S| - p^{n-1}_{|S|}(\lambda^{(j)})$.
Set $a = \lambda'_{j+1}$.
Remark~\ref{R: approach} suggests that
for each such $e_d(S)$,
we should find a homogeneous polynomial $F_{d,S} \in J_{n, \lambda}$
such that $F_{d,S} - e_d(S)$ is divisible by $x_1^{a+1}$:

\begin{lemma}
\label{L: Fds}
For each $e_d(S)$ among the generators
$J'_{n, \lambda^{(j)}}$,
we let 
$$
F_{d,S} :=
\begin{cases}
x_1^{a}e_d(S \sqcup \{1\}) &\text{if $|S| \geq n - j$,} \\
x_1^{a}e_d(S) &\text{if $|S| < n - j$.} 
\end{cases}
$$
Then $F_{d, S}$ is a homogeneous polynomial in $J_{n, \lambda}$
such that $F_{d,S} - e_d(S)$ is divisible by $x_1^{a+1}$
\end{lemma}

Our proof relies on (\ref{EQ: e}) and uses arguments similar
to Griffin's arguments in section $3$ of~\cite{G}.

\begin{proof}
Clearly, $F_{d, S}$
is homogeneous. 
Then observe that 
$$
p_{m}^{n-1}(\lambda^{(j)}) = p_{m+1}^n(\lambda^{(j)}) =
\begin{cases}
p_{m+1}^n(\lambda) - 1 &\text{if $n - m \leq  j$,} \\
p_{m+1}^n(\lambda) &\text{if $n - m > j$.} 
\end{cases}
$$

If $|S| \geq n - j$,
we know $d > |S| - p_{|S|}^{n-1}(\lambda^{(j)}) = (|S| + 1) - p_{|S|+1}^{n}(\lambda)$.
Thus, 
$d + j$
$e_d(S \sqcup \{1\})$ is in $J_{n,\lambda}$,
so is $F_{d,S}$.
By (\ref{EQ: e}), 
$$F_{d,S} = x_1^{a}e_d(S) + x_1^{a + 1}e_{d-1}(S).$$
Thus, $F_{d, S}  - x_1^{a}e_d(S)$ is divisible 
by $x_1^{a + 1}$.

If $|S| < n - j$,
we clearly know $F_{d, S}  - x_1^{a}e_d(S)$ is divisible 
by $x_1^{a + 1}$.
It remains to check $F_{d,S} \in J_{n, \lambda}$.
By repeatedly applying (\ref{EQ: e}),
we have
$$
F_{d,S} = x_1^{a} e_d(S)
= \sum_{k = 1}^{a} (-1)^{k - 1}
x_1^{a - k}
e_{d + k}(S \sqcup \{1\}) + (-1)^{a} e_{d + a}(S).
$$
Then we can show each term on the right hand side is in $J_{n, \lambda}$.
We know $d 
> |S| - p_{|S|}^{n-1}(\lambda^{(j)}) 
= |S| - p_{|S|+1}^{n}(\lambda)$.
Then for $k\geq 1$, $d + k > (|S| + 1) - p_{|S|+1}^{n}(\lambda)$,
so $e_{d+k}(S\sqcup \{1\})$
is in $J_{n,\lambda}$.
Recall that $a = \lambda_{j+1}'$.
The assumption $|S| < n - j$ implies
$j +1\leq n - |S|$,
so $a = \lambda_{j+1}' \geq \lambda_{n - |S|}'$.
We have
$$d + a > 
|S| - p_{|S|+1}^{n}(\lambda) + a 
\geq
|S| - p_{|S|+1}^{n}(\lambda) + \lambda_{n - |S|}'
= |S| - p_{|S|}^n(\lambda),$$
so $e_{d + a}(S)$
is in $J_{n, \lambda}$.
\end{proof}

Now take $\alpha \in D_{n,\lambda}$.
By Proposition~\ref{P: D},
$\alpha_1 = \lambda_{j+1}'$
for some $j$ such that $\lambda^{(j)}$ is defined.
Moreover, $(\alpha_2, \cdots, \alpha_n)
\in D_{n, \lambda^{(j)}}$.
By recursion, 
we let $f$ be the polynomial
in $G_{n-1, \lambda^{(j)}}$
with $\leading_<(f) = x^{(\alpha_2, \cdots, \alpha_n)}$.
Obtain $\tilde{f}$ from $f$ by shifting
all $x_i$ to $x_{i+1}$.
We know $\tilde{f} \in J_{n, \lambda^{(j)}}$,
so we may write $\tilde{f}$
as $\sum_{d,S}g_{d,S} e_d(S)$,
where $g_{d, S} \in \mathbb{Z}[x_2, \cdots, x_n]$.
By Remark~\ref{R: approach}
and Lemma~\ref{L: Fds},
if we let $F := \sum_{d,S}g_{d,S} F_{d,S}$,
then $F$ is a homogeneous polynomial
in $J_{n,\lambda}$
with $\leading_<(F) = x^\alpha$
and leading coefficient $1$.

\begin{thm}
The $G_{n, \lambda}$ we construct is a 
Gr\"obner basis of $I_{n, \lambda, \infty}$.
Its polynomials have integer coefficients
with leading coefficient $1$.
\end{thm}
\begin{proof}
By the recursive construction, 
each polynomial in $G_{n,\lambda}$
has integer coefficients and leading coefficient $1$.
We know $G_{n, \lambda} \subseteq J_{n, \lambda}$,
so $G_{n, \lambda} \subset I_{n, \lambda, \infty}$.
We know $\{\leading_<(F): F \in G_{n,\lambda}\} = \{B_{n, \lambda}\}$.
By Proposition~\ref{P: D},
$\langle \{\leading_<(F): F \in G_{n,\lambda}\}\rangle = 
\langle A_{n, \lambda, 
\infty}^C\rangle$.
Then the proof is finished by 
Lemma~\ref{L: main}.
\end{proof}

\begin{exa}
Suppose $n = 5$ 
and $\lambda = (2,2,1)$.
We would like to construct $F \in J_{5, (2,2,1)}$
with leading monomial $x^{(0,1,2,0,0)}$.
Here, the weak composition $(0,1,2,0,0) \in D_{5, (2,2,1)}$ comes from the following element of $\OP_{5, (2,2,1)}$:
$$
\begin{ytableau}
\none & \none & \none[3]\cr
 & 5 & 2  \cr
4 & 1\cr
\end{ytableau}
$$

We construct $F$ recursively. 
To make our computation clear, 
we reverse the recursive process.
\begin{itemize}
\item We can find $$e_0(\{1\}) \in J_{1, (2)}$$
with leading monomial $x^{(0,0)}$.
\item Based on the previous polynomial,
we can find $$e_0(\{1,2\}) \in J_{2, (2,1)}$$
with leading monomial $x^{(0,0)}$.
\item Based on the previous polynomial,
we can find $$x_1^2 e_0(\{2,3\} \in J_{3, (2,1)}$$
with leading monomial $x^{(2,0,0)}$.
It can be written as 
$$x_1e_1(\{1,2,3\}) - e_2(\{1,2,3\}) + e_2(\{2,3\}).$$
\item Based on the previous polynomial,
we can find 
$$x_1x_2e_1(\{1, 2,3,4\}) - x_1e_2(\{1,2,3,4\}) + x_1e_2(\{3,4\})\in J_{4, (2,1,1)}$$
with leading monomial $x^{(1,2,0,0)}$.
It can be written as 
$$x_1x_2e_1(\{1, 2,3,4\}) - x_1e_2(\{1,2,3,4\}) + e_3(\{1, 3,4\}).$$
\item Based on the previous polynomial,
we can find 
$$x_2x_3e_1(\{1, 2,3,4,5\}) - x_2e_2(\{1,2,3,4,5\}) + e_3(\{1, 2, 4,5\})\in J_{5, (2,2,1)}$$
with leading monomial $x^{(0,1,2,0,0)}$.
\end{itemize}
\end{exa}

\section{Acknowledgments}
We are grateful to 
Brendon Rhoades 
for carefully reading an earlier
version of this paper 
and giving many useful suggestions.
We are grateful to Sean Griffin and 
Brendon Rhoades for helpful conversations. 

\bibliographystyle{alpha}
\bibliography{main.bbl}{}
\end{document}